\documentclass[letterpaper, 10 pt, conference]{ieeeconf}  

\IEEEoverridecommandlockouts                              

\usepackage{cite}
\usepackage{amsmath,amssymb,amsfonts,bbm}
\usepackage{algorithmic}
\usepackage{graphicx}
\usepackage{textcomp}
\usepackage{xcolor}
\usepackage[ruled,vlined]{algorithm2e}
\usepackage{accents}
\pdfoutput=1
\usepackage{hyperref}
\hypersetup{
    colorlinks=true,
    linkcolor=black,
    citecolor=black,
    filecolor=black,
    urlcolor=black,
}

\newtheorem{theorem}{Theorem}

\newtheorem{lemma}{Lemma}

\newtheorem{proposition}{Proposition}

\newtheorem{assumption}{Assumption}

\newcommand{\bmat}[1]{\begin{bmatrix}#1\end{bmatrix}}
\newcommand{\norm}[1]{\lVert{#1}\rVert}
\newcommand{\R}{\mathbb{R}}	
\DeclareMathOperator*{\argmax}{arg\max}
\newcommand{\tp}{\mathsf{T}}

\title{\LARGE \bf
 Convex Programs and Lyapunov Functions for Reinforcement Learning: A Unified Perspective on the Analysis of Value-Based Methods
}

\author{Xingang Guo and Bin Hu
\thanks{Xingang Guo and Bin Hu are with the Coordinated Science Laboratory (CSL) and the Department of Electrical and Computer Engineering, 
        University of Illinois at Urbana-Champaign. Email:
        {\tt\small \{xingang2,~binhu7\}@illinois.edu}}%
}

\begin{document}

\maketitle
\thispagestyle{empty}
\pagestyle{empty}

\begin{abstract}
Value-based methods play a fundamental role in Markov decision processes (MDPs) and reinforcement learning (RL). In this paper, we present a unified control-theoretic framework for analyzing valued-based methods such as value computation (VC), value iteration (VI), and temporal difference (TD) learning (with linear function approximation). Built upon an intrinsic connection between value-based methods and  dynamic systems, we can directly use existing convex  testing conditions in control theory  to derive various convergence results for the aforementioned value-based methods. These testing conditions are convex programs in form of either linear programming (LP) or semidefinite programming (SDP), and can be solved to construct Lyapunov functions in a straightforward manner.  Our analysis reveals some intriguing connections  between feedback control systems and RL algorithms.  It is our hope that such connections can inspire more work at the intersection of system/control theory and RL.
\end{abstract}

%

\section{Introduction}

Over the past 10 years, many research ideas have emerged from the fields of control, optimization, and machine learning.
 A big research focus is on fundamental connections between control systems and iterative algorithms.
The research on this topic has led to  exciting new results on algorithm analysis and design.
 For example, iterative optimization methods have been analyzed as feedback control systems~\cite{Lessard2014, nishihara2015,hu17a,fazlyab2017analysis,sundararajan2017robust,hu2017control,hu17b,hatanaka2018passivity,hu2018dissipativity,han2019systematic, aybat2018robust,xiong2020analytical,mohammadi2020robustness,hu2021analysis,gannot2021frequency}, and control-theoretic tools have been leveraged  to design new optimization algorithms in various settings~\cite{van2017fastest,cyrus2018robust,fazlyab2018design, nelson2018integral,aybat2019universally,michalowsky2020robust,sundararajan2020analysis}. Recently, there has been an attempt to extend such control perspectives to reinforcement learning (RL). In~\cite{hu2019characterizing}, a fundamental connection between temporal difference learning and Markovian jump linear systems (MJLS) has been established. In \cite{lee2019unified}, the switching system theory has been combined with the ODE method~\cite{borkar2009stochastic, borkar2000ode} to analyze the asymptotic convergence of $Q$-learning. 
More recently, value iteration has also been connected to PID control \cite{Farahmand2021}.
Our paper is inspired by these prior results, and establishes  a new connection between RL and control theory. Specifically, we tailor various convex testing conditions in control theory for unifying the analysis of
 value-based algorithms.

RL refers to a collection of techniques for solving Markov decision processes (MDPs), and has shown great promise in many sequential decision making tasks \cite{Puterman2014, sutton2018reinforcement,bertsekas1996neuro}. Value-based methods including value computation (VC), value iteration (VI), and temporal difference (TD) learning \cite{sutton2018reinforcement} have played a fundamental role in modern RL.  The convergence proofs for these methods are typically derived in a case-by-case manner \cite{bertsekas1996neuro,sutton2018reinforcement,Puterman2014,bhandari2018finite,srikant2019finite,xu2019two,sun2020finite,xu2020finite,zhang2021finite,doan2021finiteb}.
Such case-by-case analysis may not be easily generalized. For example,  the convergence proof for VI is based on applying the contraction mapping theorem, and requires identifying the right distance metric via deep expert insights. The same distance metric may not be directly used in analyzing other algorithms such as TD learning.

In this paper, we present a unified control-theoretic framework for the convergence analysis of value-based methods. 
A key observation is that value-based methods can be viewed as dynamical control systems whose behaviors can be directly analyzed using convex programs.  In this paper, VC is modeled as a linear time invariant (LTI) positive system, and VI is viewed as a switched positive affine system. In addition, we also borrow the Markovian jump linear system (MJLS) perspective on TD learning from \cite{hu2019characterizing}. Notice that there exist many convex testing conditions for analyzing  LTI positive systems  \cite{Farina2011Pos_sys,Rantzer2011}, switched positive systems \cite{Blanchini2015, Liu2009, Mason2007, Xu2018, Fornasini2011,pastravanu2014max}, and MSLS  \cite{costa2006, costa1993stability, el1996robust, Fang2002, ji1991stability, Seiler2003bounded}.  
We show that  valued-based methods can be analyzed by directly applying the existing linear programming (LP) or semidefinite programming (SDP) conditions from the positive system or MJLS theory.  Importantly,  we can solve these convex conditions analytically to build our Lyapunov-based proofs in a more transparent manner.  It is our hope that the proposed framework can inspire more work at the intersection of system/control theory and RL.

Our analysis makes direct use of existing convex programs in control theory, and complements
the work in \cite{lee2019unified, hu2019characterizing, srikant2019finite} which rely on other types of stability analysis tools.
 There are many other convex conditions in control theory, and our work opens the possibility of re-examining these conditions in the context of RL.
Compared with \cite{hu2019characterizing}, our SDP approach has led to new stepsize bounds for TD learning. This result will be given in Section~\ref{sec: TD}.

\section{Preliminaries  and Problem Formulation}
\label{sec:BACKGROUND}

\subsection{Notation}
The set of $n$-dimensional real vectors is denoted as
$\mathbb{R}^n$. The set of $m\times n$ real matrices is denoted as $\mathbb{R}^{m \times n}$.
We use $\R_+^n$ to denote 
the set of the $n$-dimensional real vectors whose entries are all non-negative. 
 For $x \in \mathbb{R}^n$, we denote its $i$-th element as $x(i)$.  The inequality $x > 0~ (x \geq 0)$ means that $x(i) > 0~(x(i) \geq 0)$ for all $i$.  
 For $A \in \R^{n \times n}$, the inequality $A > 0~ (A \geq 0)$ means all the entries of $A$ are positive (non-negative).
 We use $A^\tp$ and $\rho(A)$ to denote the transpose and the spectral radius of $A$, respectively.   A matrix $A$ is said to be Schur stable if  $\rho(A) < 1$.
The inequality $G \succ 0~ (G \succeq 0)$ means that the matrix $X$ is positive (semi-)definite.

%

\subsection{Markov Decision Process and Reinforcement Learning}
First, we present some background materials on MDPs and RL.
Many decision making tasks can be formulated as MDPs. Consider a MDP defined by the  tuple $(\mathcal{S},\mathcal{A},P,R,\gamma)$, where $\mathcal{S}$ is the set of states,  $\mathcal{A}$ is the set of actions, $P$ is the transition kernel, $R$ is the reward function, and $\gamma\in (0,1)$ is the discount factor. In this paper, both $\mathcal{S}$ and $\mathcal{A}$ are assumed to be finite. Without loss of generality, we assume $\mathcal{S}=\{1,2,\ldots,n\}$ and $\mathcal{A}=\{1,2,\ldots, l\}$.  The transition kernel is specified by
 $P((s,a),s') = \mathbf{P}(s_{k+1}=s'|s_k = s, a_k = a)$.

A policy is a feedback law mapping from states to actions. A policy can be stochastic  and maps each state to a probability distribution over $\mathcal{A}$. The goal is to find an optimal policy that maximizes the total accumulated rewards:
$$\pi^* = \argmax_\pi \mathbb{E}\left[\sum_{k=0}^{\infty} \gamma^k R(s_k, a_k)\big\vert a_k\sim\pi(\cdot |s_k), s_0\right].$$

To obtain an optimal policy, one can solve the optimal value function $J^*$ from the optimal Bellman equation: 
 \begin{align}
 J^*(s)=\max_{a\in \mathcal{A} }\left(R(s,a)+\gamma \sum_{s'\in \mathcal{S}} P((s,a),s') J^*(s')\right).
 \end{align}
Once $J^*$ is found,  one can construct the optimal policy as
\begin{align*}
\pi^*(s)=\argmax_{a\in \mathcal{A} }\left(R(s,a)+\gamma \sum_{s'\in \mathcal{S}} P((s,a),s') J^*(s')\right).
\end{align*}
The optimal Bellman equation depends on the transition kernel. If the transition model is  unknown, 
 RL methods (e.g. TD learning, $Q$-learning, policy gradient, etc) can be applied.

\subsection{Value Computation}
The performance of a given policy $\pi$ can be evaluated from the associated value function $J_\pi$, which is defined as
 $$J_\pi(i) =  \mathbb{E}\left[\sum_{k=0}^{\infty} \gamma^k R(s_k, a_k)\big\vert a_k\sim\pi(\cdot |s_k), s_0=i\right].$$
For given $\pi$, denote the probability transition matrix of $\{s_k\}$  as $P_\pi$. 
 Then $J_\pi$ can be solved from the Bellman equation:
\begin{align}
J_{\pi}(i) = R_{\pi}(i) + \gamma \sum_{j} P_\pi(i,j) J_{\pi}(j),
\end{align}
where $P_\pi(i,j)$ is the $(i,j)$-th entry of $P_\pi$, and $R_\pi(i)$ is the immediate reward obtained from state $i$ under the policy $\pi$. 
 The above Bellman equation can be compactly rewritten as
 \begin{align} \label{eq: PEopt}
 J_\pi =R_\pi+ \gamma P_\pi J_\pi .
 \end{align}
 Obviously, $J_\pi$ can be calculated as $J_\pi=(I-\gamma P_\pi)^{-1} R_\pi$ for any $0<\gamma<1$.
 To avoid matrix inversion, a popular approach for solving $J_\pi$ is to apply the following iterative value computation (VC) scheme:
  \begin{align}
  \label{eq:VC1}
 J_{k+1} =\gamma P_\pi J_{k} +R_\pi .
 \end{align}
%
It is known that the above method is guaranteed to converge to $J_\pi$ at a linear rate $\gamma$. This is actually obvious from the linear system theory. For the right stochastic matrix $P_\pi$, we have $\rho(P_\pi) = 1$. Hence the convergence of \eqref{eq:VC1} can be guaranteed by  the fact that we have $\rho(\gamma P_\pi) = \gamma \in (0, 1)$.

\subsection{Value Iteration}
One can solve the optimal value function $J^*$ by recursively applying the Bellman operator $T(\cdot): \mathbb{R}^n \to \mathbb{R}^n$. This leads to the famous value iteration (VI) method which iterates as 
$${J}_{k+1}(s) = \max_{a\in \mathcal{A}}\left(R(s,a) + \gamma \sum_{s'} P((s,a),s')J_k(s')\right).$$ A pseudo-code for VI is provided as follows.

\begin{algorithm} \label{VI_algorithm}
\SetAlgoLined
Initialization: $ J(s) \leftarrow J_0(s),  \forall ~ s \in \mathcal{S}$\;
 Repeat \\ 
  For all $s \in \mathcal{S}$ \;
  ${J}(s) \! \leftarrow \! \max_{a \in \mathcal{A}}\left(R(s,a) + \gamma \sum_{s'} P((s,a),s')J(s')\right)$\;
  Until $J$ converge 
 \caption{Value iteration algorithm}
\end{algorithm}
The iteration of VI can be compactly rewritten as
\begin{equation}
{J}_{k+1} = T({J}_{k}),  \label{recursion_VI}
\end{equation}
where $T(\cdot)$ is the Bellman optimality operator.  
It is known that VI converges to $J^*$ at the rate $\gamma$, and a standard way to prove this is to
apply the contraction mapping theorem \cite{Puterman2014}.

\begin{table*}[h]
\caption{Convex Programs for Value-based Methods}
\label{Main_table}
\begin{center}
  \begin{tabular}{c|c|c|c}
   \hline
    Value-based Algorithms & Type of Dynamic Systems & Convex Programs  &Lyapunov functions  \\ 
    \hline 
    Value Computation & LTI Positive system Eq.\eqref{eq:VC2} & LP \& SDP (Theorem \ref{Thm: PE})  & Eq. \eqref{LF_PE}\\
    Value Iteration  & Switched positive system Eq.\eqref{eq:VI_switch} & LP (Condition \eqref{mc2})  & Eq. \eqref{eq:VVI} \\
    TD(0) with Linear Function Approximation  &MJLS Eq.\eqref{MJLS_gen} &SDP (Proposition \ref{MSS}) & Eq. \eqref{eq:MJLSLya}\\
    \hline
  \end{tabular}
\end{center}
\end{table*}

\subsection{TD Learning with Linear Function Approximation}

It is very common that the transition kernel of MDP is unknown. In this case, the VC scheme \eqref{eq:VC1} is not applicable. Instead, TD learning can be used to estimate the value function from sampled trajectories of the underlying Markov chain $\{ s_k \}$. Most applications have enormous state spaces, making policy evaluation difficult. Then one needs to incorporate function approximation techniques. Suppose the value function is estimated as $J_\pi(s)\approx \phi(s)^\tp \theta_\pi$ where $\phi$ is the feature vector and $\theta_\pi$ is the weight to be estimated.   One model-free way to estimate $\theta_\pi$ is to apply the following TD(0)~recursion: 
\begin{equation*} 
    \theta_{k+1} = \theta_{k} - \alpha \phi(s_k)\left( (\phi(s_k) - \gamma \phi(s_{k+1}))^\tp \theta_k - R_{\pi}(s_k) \right), 
    \end{equation*}  
where $R_\pi$ is the reward, $\gamma$ is the discount factor, and $\alpha$ is the learning rate. 
The Markov nature of $\{s_k\}$ has caused trouble for  the finite time analysis of the above method. 
Very recently, various specialized tricks \cite{bhandari2018finite,srikant2019finite,hu2019characterizing} have been developed to address this technical difficulty, leading to several useful finite time results for TD(0) with sufficiently small $\alpha$.

%



\subsection{Main Objective: Unified Analysis of VC, VI, and TD}
The objective of this work is to develop a simple routine unifying the analysis of VC, VI, and TD(0) with linear function approximation. 
Built upon the connections between value-based methods and dynamic systems,  we can directly use  existing convex programs (LP/SDP) in control theory to analyze the above value-based methods. In addition, these convex programs lead to different types of Lyapunov functions, making the Lyapunov-based convergence analysis transparent. Table \ref{Main_table} summarizes our main results in this work. Our analysis sheds new light on how to combine convex programs and Lyapunov analysis in the context of RL. 


\section{Unified Analysis of Value-Based Methods}
\label{sec:mu_VC}

\subsection{LPs and SDPs for Analyzing VC}
To analyze VC, we apply  \eqref{eq: PEopt} and \eqref{eq:VC1} to rewrite VC as
\begin{align}\label{eq:VC2}
\zeta_{k+1} = A_\pi \zeta_k 
\end{align}
where  $\zeta_k=J_k-J_\pi$, and $A_\pi = \gamma P_\pi \in \mathbb{R}^{n \times n}$.  Notice that \eqref{eq:VC2} is actually a positive system since  $A_\pi = \gamma P_\pi \geq 0$. To verify the Schur stability of $A_\pi$, the following convex conditions for positive linear systems can be directly applied.
\begin{proposition} \label{Propo_PE}
Suppose $A_\pi \geq 0$. Then each of the following conditions provides a necessary and sufficient condition for the stability of \eqref{eq:VC2}:
\begin{enumerate}
\item \label{c2}  $\exists~\xi \in \mathbb{R}^n$ s.t.  $\xi > 0, ~\text{and}~ A_\pi \xi -  \xi<0$,
\item \label{c3} $\exists~\nu \in \mathbb{R}^n$ s.t. $\nu > 0, ~\text{and}~  \nu^\tp A_\pi -  \nu^\tp<0$,
\item \label{c4} $\exists~ G \in \mathbb{R}^{n\times n}$ s.t. $G \succ 0$  $ A_\pi^\tp G A_\pi \prec  G$.
\end{enumerate}
\end{proposition}
\begin{proof}
This result is just the discrete-time counterpart of Proposition 1 in \cite{Rantzer2011} and can be proved similarly. 
\end{proof}
\vspace{0.05in}

The above convex programs can be solved to obtain three types of Lyapunov functions for \eqref{eq:VC2}.  Let $1_n$ denote the $n$-dimensional 
vector whose entries are all equal to $1$.  Then our results can be stated as follows. 

\begin{theorem} \label{Thm: PE}
Consider the recursion \eqref{eq:VC2}. Set $(\xi,\nu,G)$ as 
$$\xi = 1_n , \,\,\,\nu = \omega,\,\,\, G = \operatorname{diag}\left(\frac{\nu(1)}{\xi(1)},\cdots, \frac{\nu(n)}{\xi(n)}\right),$$ where $\omega$ is the stationary distribution of $P_\pi$. Then we have
\begin{align}
&  \xi > 0~ \text{and}~ A_\pi \xi \leq \gamma \xi, \label{mmc2}  \\
&  \nu > 0 ~ \text{and}~ \nu^\tp A_\pi  \leq  \gamma \nu^\tp,  \label{mmc3} \\
&   G\succ 0  ~ \text{and}~ A_\pi ^\tp G A_\pi  \preceq  \gamma^2 G.  \label{mmc4}
\end{align}
This leads to the following three types of Lyapunov~functions
\begin{equation} \label{LF_PE}
V_1(\zeta) = \max_i|\zeta(i)|, \quad V_2(\zeta)=|\nu^\tp \zeta|, \quad V_3(\zeta) = \zeta^\tp G \zeta,
\end{equation}
which satisfy $V_1(\zeta_k) \leq C_1 \gamma^k$, $V_2(\zeta_k)\le C_2 \gamma^k$, and $V_3(\zeta_k)\le C_3 \gamma^{2k}$ for some fixed positive constants $(C_1,C_2,C_3)$.
\end{theorem}
\begin{proof}
Since $P_\pi$ is always a right stochastic matrix, we have $ P_\pi  1_n =1_n $ and $\omega^\tp P_\pi = \omega^\tp$. Therefore,  
\begin{align*}
&A_\pi 1_n -  \gamma 1_n = \gamma ( P_\pi  1_n  -   1_n) =  0 \\
&\omega^\tp A_\pi - \gamma \omega^\tp =  \gamma\left( \omega^\tp P_\pi -   \omega^\tp \right)= 0.
\end{align*}
Hence \eqref{mmc2} and \eqref{mmc3} hold. The third condition can be proved as discussed by Proposition 2 in \cite{Rantzer2011}. The rest of the results follow from standard arguments in positive system theory.
\end{proof}

Notice that the conditions \eqref{mmc2} \eqref{mmc3} are LPs, and \eqref{mmc4} can be solved as SDPs.
Theorem \ref{Thm: PE} provides three different types of Lyapunov functions for the positive system \eqref{eq:VC2}: 
\begin{enumerate}
\item {$\ell_\infty$-type}: $V_1(\zeta) = \max_i |\zeta(i)|$; 
\item {Linear-type copositive}: $V_2(\zeta) =|\nu^\tp \zeta|$; 
\item {Quadratic}: $V_3(\zeta) = \zeta^\tp G \zeta$. 
\end{enumerate}
It is trivial to show $V_1(\zeta_{k+1})\le \gamma V_1(\zeta_k)$ and $V_3(\zeta_{k+1})\le \gamma^2 V_3(\zeta_k)$. For $i=2$, 
the Lyapunov function is copositive and works slightly differently. Here we briefly explain how it works. 
 For any $\zeta_0\in \R^n$, $\exists \zeta_0^+, \zeta_0^-\in \R_+^n$ s.t. $\zeta_0=\zeta_0^+ -\zeta_0^-$. Let $\{\zeta_k^+\}$ and $\{\zeta_k^-\}$ be the state trajectories of \eqref{eq:VC2} initialized from $\zeta_0^+$ and $\zeta_0^-$, respectively. 
By linearity, we have $\zeta_k=\zeta_k^+ -\zeta_k^-$.
Based on the condition \eqref{mmc3}, we can show $V_2(\zeta_k)\le V_2(\zeta_k^+)+V_2(\zeta_k^-)\le \gamma^k (V_2(\zeta_0^+)+V_2(\zeta_0^-))$. This ensures the convergence of VC.

A key message from the above analysis is that the Lyapunov function construction for positive linear systems can be simpler than general LTI systems due to the use of LPs. 
Since the construction of the max-type (or $\ell_\infty$-type) Lyapunov function $V_1(\cdot)$ is independent of the underlying policy, we may construct an $\ell_\infty$-type common Lyapunov function for cases where the policy is changing over time.

\subsection{LPs and Common Lyapunov Functions for VI}
\label{sec:VIP}
Next we establish the connection between VI and switched positive affine  systems. This will lead to LP conditions for analyzing VI. The VI scheme $J_{k+1}=T(J_k)$ can be recast~as
\begin{align}
\label{eq:VI_switch}
J_{k+1}=\gamma P_{\sigma_k} J_k+R_{\sigma_k}
\end{align}
where $\sigma_k\in \{1,2,\ldots, l^n\}$. Recall $l$ and $n$ denote the size of action space and state space, respectively. When $\sigma_k=m$, we set $P_{\sigma_k}=P_m$ and $R_{\sigma_k}=R_m$. For all $m\in\{1,2,\cdots,l^n\}$,  $P_m$ is an $n\times n$ matrix whose $i$-th row (for all $i$) is a row vector in the form of $\bmat{P((i,a),1)& P((i,a),2) & \ldots & P((i,a), n)}$ with some $a\in \mathcal{A}$. Similarly, for all $m$, the vector $R_m$ is a vector whose $i$-th element (for all $i$) is equal to $R(i,a)$ for some $a\in \mathcal{A}$. The total number of the $(P_m, R_m)$ pairs is $l^n$, and we denote the set of all such pairs as $\mathbf{\Lambda}$. Therefore,
we can just view \eqref{eq:VI_switch} as a switched positive affine system, and  it is not that surprising that we can analyze VI via switched system theory.

For ease of exposition, we first address the case where  $R_m = 0$ for all $m$.
 In this case, we have $J^*=0$, and \eqref{eq:VI_switch}
 can be rewritten as a  switched positive  linear system:
 \begin{align}
\label{eq:VI_switch_R_0}
J_{k+1}=A_{m} J_k, \quad m \in \{1 ,2, \cdots, l^n \},
\end{align}
where $A_m = \gamma P_m \in \mathbb{R}^{n \times n}$. A well-known fact is that the system state of \eqref{eq:VI_switch_R_0} may diverge for some switching sequence even when $A_{m}$ is Schur stable for all $m$ \cite{DeCarlo2000}. 
The stability guarantees  for \eqref{eq:VI_switch_R_0} are typically obtained by extending the Lyapunov approach presented in Proposition~\ref{Propo_PE}. One way is to use the {common Lyapunov function (CLF)}.
\begin{proposition} \label{Pro_switch}
Suppose $A_m\ge 0$ for all $m$. Then each of the following conditions provides a sufficient condition for the stability of the switched positive system \eqref{eq:VI_switch_R_0}:
\begin{enumerate}
\item $\exists \xi \in \mathbb{R}^n$ s.t. $\xi > 0$ and $A_m\xi - \xi<0$ for all $A_m$. \label{d2} 
\item $\exists \nu \in \mathbb{R}^n$ s.t. $\nu > 0$ and $\nu^\tp A_m -\nu^\tp<0$ for all $A_m$. \label{d3}
\item $\exists$ a  matrix $G\succ 0$ s.t. $A_m^\tp G A_m-G\prec 0$ for all $A_m$.  \label{d4}
\end{enumerate}
\end{proposition}
\begin{proof}
The proof is standard. The third condition actually does not require $A_m$ to be positive. 
If $A_m^\tp G A_m-G\prec 0$, then $\exists$ $\epsilon>0$ such that $A_m^\tp G A_m-(1-\epsilon) G\preceq 0$. Hence we can define a Lyapunov function $V(J_k)=J_k^\tp G J_k$ satisfying $V(J_{k+1})\le (1-\epsilon) V(J_k)$. This ensures the stability of \eqref{eq:VI_switch_R_0}.
The first and second conditions do require $A_m$ to be positive, and can be proved similarly.  See \cite{Fornasini2011,pastravanu2014max} for details.
\end{proof}

Similar to Theorem \ref{Thm: PE}, the testing conditions in Proposition~\ref{Pro_switch} can be modified to analyze convergence rates of~\eqref{eq:VI_switch_R_0}. One will obtain similar rate bounds as presented in Theorem~\ref{Thm: PE} if any of the following is feasible:
\begin{align}
& \exists~  \xi \in \mathbb{R}^n~s.t. ~ \xi > 0~ \text{and}~ A_m\xi \leq \gamma \xi\,\,\forall~ m. \label{mc2}  \\
& \exists~  \nu \in \mathbb{R}^{ n} ~ s.t. ~  \nu > 0 ~ \text{and}~ \nu^\tp A_m \leq  \gamma \nu^\tp\,\,\forall~m. \label{mc3} \\
& \exists~  G \in \mathbb{R}^{n\times n} ~  s.t. ~  G\succ 0  ~ \text{and}~ A_m^\tp G A_m \preceq  \gamma^2 G \,\,\forall~ m. \label{mc4}
\end{align}
It is interesting to see that in general, the system \eqref{eq:VI_switch_R_0} does not have a common linear copositive Lyapunov function since the stationary distributions for different $P_m$ are typically not the same. It also seems difficult to construct a common solution $G$ for the SDP \eqref{mc4}. 
However, since all $P_m$ share the same right eigenvector $1_n$, we have $\gamma P_m 1_n = \gamma 1_n$ for all~$m$. Hence we can solve \eqref{mc2}  to obtain an $\ell_\infty$-type CLF:
\begin{align}\label{eq:VVI}
V(J_k) = \|J_k-J^*\|_\infty.
\end{align}
Condition \eqref{mc2} can be used to guarantee $V(J_k)\le \gamma^k V(J_0)$.


Now, we can extend the above analysis to the general case where $R_m \neq 0$. In this case, we will show that the iterations of VI can be upper and lower bounded by the trajectories of two stable positive linear systems. Hence positive system theory can still be applied. We need the following lemma.
\begin{lemma}
\label{lem1}
Consider the switched positive affine system~\eqref{eq:VI_switch} with a switching sequence $\{\sigma_k\}$ completely determined by the Bellman optimality operator\footnote{In other words, the trajectory of such a switched system now exactly matches the sequence generated by the VI method.}. Then the following inequality holds for all $k$
\vspace{-0.05in}
\begin{align}
\gamma P^*(J_k-J^*) \le J_{k+1}-J^*\le \gamma P_{\sigma_k} (J_k-J^*).
\end{align}
  \end{lemma}
  \vspace{0.1in}
 \begin{proof}
Since $J^*=T(J^*)$, there exists a pair $(P^*, R^*)\in \mathbf{\Lambda}$ such that $J^*= \gamma P^* J^*+R^*$. By the definition of the Bellman optimality operator, one can show that the following two inequalities holds for all $m$:
\begin{align}
\gamma P_m J^*+R_m &\le \gamma P^* J^*+R^*, \label{eq:con1}  \\
 \gamma P_m J_k+R_m &\le \gamma P_{\sigma_k} J_k+R_{\sigma_k}. \label{eq:con2}
\end{align}
Using \eqref{eq:con1}, \eqref{eq:con2}, and the fact that $J^*=\gamma P^* J^*+R^*$, one can verify
$\gamma P^* (J_k-J^*) \le J_{k+1} - J^* \le  \gamma P_{\sigma_k} (J_k-J^*)$.
This leads to the desired conclusion. 
\end{proof}  

Based on Lemma \ref{lem1}, we obtain the following main result.
 \begin{theorem}
Consider the switched positive affine system~\eqref{eq:VI_switch} with a switching sequence $\{\sigma_k\}$ completely determined by the Bellman optimality operator $T$. Suppose the sequence $\{J_k^u\}$ is generated by the system $J_{k+1}^u-J^*=\gamma P_{\sigma_k} (J_k^u-J^*)$ with the same switching sequence $\{\sigma_k\}$. Let the sequence $\{J_k^{o}\}$ be generated by the system $J_{k+1}^o-J^*=\gamma P^*(J_k^o-J^*)$. Suppose $J_0=J_0^u=J_0^o$. Then we have
 \begin{align}\label{eq:keybound1}
J_k^o-J^*\le J_k-J^*\le J_k^u-J^*,\,\forall k
\end{align}
 \end{theorem}
 \vspace{0.05in}
\begin{proof}
This theorem can be proved using induction. When $k=0$, it is straightforward to verify that \eqref{eq:keybound1} holds as a consequence of Lemma \ref{lem1}. Suppose \eqref{eq:keybound1} holds for $k=t$. For $k=t+1$, we can apply Lemma \ref{lem1} to show 
\begin{align*}
J_{t+1}-J^*&\le \gamma P_{\sigma_k} (J_t-J^*) \le \gamma P_{\sigma_k} (J_t^u-J^*)=J_{t+1}^u-J^*
\end{align*}
where the second step follows from the fact that $P_{\sigma_k}$ is right stochastic.
 Based on Lemma \ref{lem1},  we can use a similar argument to show $J_{t+1}-J^*\ge J_t^o-J^*$. Hence \eqref{eq:keybound1} holds for $k=t+1$. This completes the proof.
\end{proof}
 
 Therefore, we can directly apply the LP condition \eqref{mc2} to construct an $\ell_\infty$-type CLF for VI and prove the rate bound
$\|J_k-J^*\|_\infty\le \max\{\|J_k^u-J^*\|_\infty, \|J_k^o-J^*\|_\infty\} \le \gamma^k \|J_0-J^*\|_\infty$.
This demonstrates how to apply the 
simple LP condition \eqref{mc2} to analyze VI.


\subsection{SDPs for TD(0) with Linear Function Approximation}
\label{sec: TD}
In this section, we provide SDP-based finite time analysis for TD(0) with linear function approximation. Since TD(0)  can be viewed as a MJLS, the SDP-based stability conditions for MJLS can be directly applied. 
Recall that TD(0) (with linear function approximation) follows the update rule $\theta_{k+1} = \theta_k - \alpha \phi(s_k)\left( (\phi(s_k) - \gamma \phi(s_{k+1}))^\tp \theta_k - R_{\pi}(s_k) \right)$, where $\phi$ is the feature vector, and $\theta$ is the weight to be estimated. We can augment $\bmat{s_{k+1}^\tp \,\, s_{k}^\tp}^\tp\in \mathcal{S}\oplus \mathcal{S}$ as a new vector~$z_k$. Obviously, there is a one-to-one mapping from $\mathcal{S}\oplus \mathcal{S}$ to the set $\mathcal{N}=\{1,2,\cdots, n^2\}$. Without loss of generality, $\{z_k\}$ can be set up as a Markov chain sampled from $\mathcal{N}$.
 Suppose $\theta_\pi$  is the solution to the  projected
Bellman equation for the fixed policy $\pi$.  Due to the one-to-one correspondence between $\bmat{s_{k+1}^\tp \,\, s_{k}^\tp}^\tp$ and $z_k$, the iteration of TD(0) can be recast~as
\begin{equation} \label{TD_line}
\theta_{k+1} - \theta_\pi = \theta_k - \theta_\pi + \alpha \left( A_{z_k} (\theta_k - \theta_\pi) + b_{z_k} \right),
\end{equation}
where $A_{z_k}=\phi(s_k) ( \gamma \phi(s_{k+1}) -  \phi(s_k) )^\tp$ and  $b_{z_k} =  \phi(s_k)\left( R_{\pi}(s_k) + (\phi(s_k) - \gamma \phi(s_{k+1}))^\tp \theta^\pi\right)$.  When $z_k=i\in \mathcal{N}$, we have $A_{z_k}=A_i$ and $b_{z_k}=b_i$.
We denote $\zeta_k=\theta_k-\theta_\pi$. Then \eqref{TD_line} can be rewritten as the following MJLS:
\begin{align}\label{MJLS_gen}
\zeta_{k+1}=H_{z_k} \zeta_k+\alpha b_{z_k} u_k.
\end{align}
where $H_{z_k}=I+\alpha A_{z_k}$,  and $u_k=1$ $\forall k$.
 When $z_k=i\in \mathcal{N}$, we have $H_{z_k} = H_i $. 
Denote $p_{ij}=\mathbf{P}(z_{k+1}=j\vert z_k=i)$, and $N=n^2$. Then the following mean square stability condition  \cite{costa2006, costa1993stability, el1996robust} can be directly applied to analyze \eqref{MJLS_gen}.
\begin{proposition} \label{MSS}
The MJLS \eqref{MJLS_gen} is mean square stable (MSS) if and only if there exist matrices $G_i \succ 0$ for $i = 1, \cdots, N$ such that the following SDP is feasible:
\begin{equation} 
G_i - H_i^\tp \left(\sum_{j=1}^N p_{ij}G_j\right)H_i \succ 0, \text{for}~ i = 1, \cdots, N. \label{MSSc1}
\end{equation}
\end{proposition}
\begin{proof}
This stability condition is well known. For more details, see  discussions in \cite{costa2006} or \cite{costa1993stability}.
\end{proof}

There are multiple ways to prove the mean square stability of \eqref{MJLS_gen} from the SDP condition \eqref{MSSc1}.
One way is to construct the following quadratic Lyapunov function from  $\{G_i\}$:
\begin{align}\label{eq:MJLSLya}
V(\zeta_k)=\mathbb{E} \left[\zeta_k^\tp G_{z_k} \zeta_k\right].
\end{align}
Once the MJLS \eqref{MJLS_gen} is shown to be MSS, Theorem 3.33 in \cite{costa2006} can be applied to show that \eqref{MJLS_gen} is also asymptotically wide sense stationary, and then the mean square TD error can be exactly calculated via Proposition~3.35 in \cite{costa2006}. As a matter of fact, the convergence bounds in Corollary 2 of \cite{hu2019characterizing} can be directly applied  whenever \eqref{MJLS_gen} is MSS. 
Therefore, the finite time analysis of TD(0) boils down to checking the mean square stability of the MJLS \eqref{MJLS_gen}. Next, we show how to construct solutions for the SDP condition \eqref{MSSc1} under the following standard assumption. 
\begin{assumption}\label{assumption1}
Suppose $\{z_k\}$ is irreducible and aperiodic.  Denote $p_i^\infty=\lim_{k\rightarrow \infty}\mathbf{P}(z_k=i)$ and $\bar{A}=\sum_{i = 1}^N {p_i}^\infty A_i$. We assume $\bar{A}$ is Hurwitz, and $\sum_{i=1}^N p_i^\infty b_i=0$. 
\end{assumption}

Under Assumption \ref{assumption1},
let $\bar{G}$ be the solution to the Lyapunov equation $\bar{A}^\tp \bar{G}+\bar{G} \bar{A}=-I$. 
We also denote $X_i=A_i^\tp \bar{G}+\bar{G} A_i+I/(p_i^\infty N)$. 
Now we can state the following result.

\begin{lemma}
Supposed Assumption \ref{assumption1} is given.
For sufficiently small $\alpha$, we can solve the SDP \eqref{MSSc1} by choosing $G_i=\bar{G}+\alpha\tilde{G}_i$, where $\tilde{G}_N=0$ and $\tilde{G}_i$ (for $i=1, \cdots, N-1$) is solved from the following linear equation:
\begin{align}\label{eq:Geq}
\tilde{G}_i -\sum_{j=1}^{N-1} p_{ij} \tilde{G}_j=X_i, \,\,\,\mbox{for}\,\, \,i=1,\cdots, N-1.
\end{align}
\end{lemma}
\begin{proof}
First, notice that \eqref{eq:Geq} does have a  unique solution. To see this, let $\hat{P}\in \R^{(N-1)\times (N-1)}$ be a substochastic matrix whose $(i,j)$-th entry is equal to $p_{ij}$.  Then $\hat{P}$ is a submatrix of the transition matrix of $\{z_k\}$, and has a spectral radius which is smaller than $1$. Hence $(I_{N-1}-\hat{P})$ is invertible, and
\eqref{eq:Geq} admits a unique well-defined solution.  Since $X_i$ is symmetric for all $i$, the resultant matrices $\{\tilde{G}_i\}$ are also symmetric. Now we briefly explain our choices of $G_i$. If we substitute $G_i=\bar{G}+\alpha\tilde{G}_i$ into \eqref{MSSc1}, we get
\begin{align}\label{MSSc2}
\tilde{G}_i-\sum_{j=1}^N p_{ij}\tilde{G}_j-(A_i^\tp \bar{G}+\bar{G} A_i)+O(\alpha)\succ 0, \,\forall i
\end{align}
For $i=1,\cdots, N-1$, we can substitute  \eqref{eq:Geq} and $\tilde{G}_N=0$ into \eqref{MSSc2} to simplify it as $I/(p_i^\infty N)+O(\alpha)\succ 0$, which clearly holds for sufficiently small $\alpha$. For $i=N$, we can use the fact $\bar{A}^\tp \bar{G}+\bar{G} \bar{A}=-I$ to show 
\begin{align*}
A_N^\tp \bar{G}+\bar{G} A_N=\frac{1}{p_N^\infty}\left(-I-\sum_{i=1}^{N-1}p_i^\infty(A_i^\tp \bar{G}+\bar{G}A_i)\right)
\end{align*}
We have $A_i^\tp \bar{G}+\bar{G} A_i=\tilde{G}_i-\sum_{j=1}^{N-1}p_{ij} \tilde{G}_j-I/(p_i^\infty N)$ for $i< N$ (see \eqref{eq:Geq}). Substituting these  into \eqref{MSSc2} for $i=N$ leads to $I/(p_N^\infty N)+O(\alpha)\succ 0$,
which holds for small $\alpha$.
\end{proof}

Next, we provide an explicit upper bound on $\alpha$. To make sure that $\{\bar{G}+\alpha\tilde{G}_i\}$  solves the SDP condition \eqref{MSSc1}, we need
\begin{align}\label{eq:alpha}
\bar{G}+\alpha \tilde{G}_i \succ 0, \,\,\,I/(p_i^\infty N)+\alpha M_i+\alpha^2 \tilde{M}_i\succ 0, \,\,\, \forall i
\end{align}
where  $\tilde{M}_i=-A_i^\tp (\sum_{j=1}^N p_{ij} \tilde{G}_j) A_i$, and 
 $M_i=-A_i^\tp \bar{G} A_i-A_i^\tp (\sum_{j=1}^N p_{ij} \tilde{G}_j)-(\sum_{j=1}^N p_{ij} \tilde{G}_j) A_i$.
We know $\bar{G}\succ 0$ and $I/(p_i^\infty N)\succ 0$. Notice $\{\tilde{G}_i\}$, $\{M_i\}$, and $\{\tilde{M}_i\}$ are symmetric matrices which are completely determined by  $\{A_i\}$ and $p_{ij}$. 
Hence the SDP condition \eqref{MSSc1} is feasible with $G_i=\bar{G}+\alpha \tilde{G}_i$ if for all $i\in \mathcal{N}$, $\alpha$ satisfies $\lambda_{\min}(\bar{G})+\alpha \lambda_{\min}(\tilde{G}_i)>0$ and
\begin{align}\label{eq:Mineq}
1/(p_i^\infty N)+\alpha\lambda_{\min}(M_i)+ \alpha^2\lambda_{\min}(\tilde{M}_i)>0.
\end{align}
where $\lambda_{\min}$ denotes the smallest eigenvalue.
Let $\mathbf{1}_{\mathcal{D}}$ denote the indicator function for any set $\mathcal{D}$.
We have $\lambda_{\min}(\bar{G})+\alpha \lambda_{\min}(\tilde{G}_i)>0$ for any $0<\alpha<\frac{\lambda_{\min}(\bar{G})}{|\lambda_{\min}(\tilde{G}_i)|(1-\textbf{1}_{\tilde{G}_i\succeq 0})}$. When $\tilde{G}_i\succeq 0$, this bound becomes $+\infty$.
It is also straightforward to verify that \eqref{eq:Mineq} is true for any $0<\alpha<\bar{\alpha}_i$, 
where $\bar{\alpha}_i$ is defined as $\bar{\alpha}_i=\frac{1}{p_i^\infty N |\lambda_{\min}(M_i)|(1-\textbf{1}_{M_i\succeq 0})}$ if $\tilde{M}_i\succeq 0$, and $\bar{\alpha}_i= \frac{-\lambda_{\min}(M_i)-\sqrt{\lambda_{\min}^2(M_i)-4\lambda_{\min}(\tilde{M}_i)/(p_i^\infty N)}}{2\lambda_{\min}(\tilde{M}_i)}$ otherwise.
This leads to the following result.

\begin{theorem} \label{theorem:TD0withLA}
Given Assumption \ref{assumption1},  the TD(0) method \eqref{TD_line}
with step size
 $0<\alpha<\min_{i\in \mathcal{N}}\left\{\bar{\alpha}_i, \frac{\lambda_{\min}(\bar{G})}{|\lambda_{\min}(\tilde{G}_i)|(1-\textbf{1}_{\tilde{G}_i\succeq 0})}\right\}$
 is MSS, and the mean square estimation error $\mathbb{E}\norm{\theta_k-\theta_\pi}^2$ converges exponentially to its stationary value. 
\end{theorem}
\begin{proof} 
From the above discussion, our stepsize bound can guarantee \eqref{eq:alpha}, and hence \eqref{TD_line} is MSS. Then the convergence behavior of $\mathbb{E}\norm{\theta_k-\theta_\pi}^2$ can be shown using Proposition 3.35 of \cite{costa2006} or Corollary 2 of \cite{hu2019characterizing}.
\end{proof}

With the MSS property, we can directly apply Corollary~2 in \cite{hu2019characterizing} to obtain explicit formulas for the convergence rate and the steady state error. We skip those formulas. 
Clearly, our result is closely related to \cite{hu2019characterizing} which also analyzes TD learning using MJLS theory. A key difference is that the analysis in \cite{hu2019characterizing} boils down to an LTI system formulation without exploiting the SDP ~\eqref{MSSc1}.
Our SDP approach brings a new benefit in providing an explicit stepsize bound guaranteeing MSS,
as specified by Theorem \ref{theorem:TD0withLA}.  In contrast, the analysis in \cite{hu2019characterizing} relies on advanced eigenvalue perturbation theory and only shows that TD(0) is MSS for sufficiently small $\alpha$ without providing such explicit stepsize bounds.

\section{CONCLUSION and Future Work}
\label{sec:con}
In this paper, we show that existing convex programs in control theory can be directly used to analyze value-based methods such as VC, VI, and TD(0) with linear function approximation. It is possible that these convex programs can be extended to address the impacts of computation error and delay. This will be investigated in the future.

\section*{ACKNOWLEDGMENT}
This work is generously supported by the NSF award 
CAREER-2048168 and the 2020 Amazon research award.



\bibliographystyle{IEEEtran}
\bibliography{IEEEabrv,main}

\begin{thebibliography}{10}
\providecommand{\url}[1]{#1}
\csname url@samestyle\endcsname
\providecommand{\newblock}{\relax}
\providecommand{\bibinfo}[2]{#2}
\providecommand{\BIBentrySTDinterwordspacing}{\spaceskip=0pt\relax}
\providecommand{\BIBentryALTinterwordstretchfactor}{4}
\providecommand{\BIBentryALTinterwordspacing}{\spaceskip=\fontdimen2\font plus
\BIBentryALTinterwordstretchfactor\fontdimen3\font minus
  \fontdimen4\font\relax}
\providecommand{\BIBforeignlanguage}[2]{{%
\expandafter\ifx\csname l@#1\endcsname\relax
\typeout{** WARNING: IEEEtran.bst: No hyphenation pattern has been}%
\typeout{** loaded for the language `#1'. Using the pattern for}%
\typeout{** the default language instead.}%
\else
\language=\csname l@#1\endcsname
\fi
#2}}
\providecommand{\BIBdecl}{\relax}
\BIBdecl

\bibitem{Lessard2014}
L.~Lessard, B.~Recht, and A.~Packard, ``Analysis and design of optimization
  algorithms via integral quadratic constraints,'' \emph{SIAM Journal on
  Optimization}, vol.~26, no.~1, pp. 57--95, 2016.

\bibitem{nishihara2015}
R.~Nishihara, L.~Lessard, B.~Recht, A.~Packard, and M.~Jordan, ``A general
  analysis of the convergence of {ADMM},'' in \emph{International Conference on
  Machine Learning}, 2015, pp. 343--352.

\bibitem{hu17a}
B.~Hu and L.~Lessard, ``Dissipativity theory for {N}esterov's accelerated
  method,'' in \emph{International Conference on Machine Learning}, vol.~70,
  2017, pp. 1549--1557.

\bibitem{fazlyab2017analysis}
M.~Fazlyab, A.~Ribeiro, M.~Morari, and V.~M. Preciado, ``Analysis of
  optimization algorithms via integral quadratic constraints: Nonstrongly
  convex problems,'' \emph{SIAM Journal on Optimization}, vol.~28, no.~3, pp.
  2654--2689, 2018.

\bibitem{sundararajan2017robust}
A.~Sundararajan, B.~Hu, and L.~Lessard, ``Robust convergence analysis of
  distributed optimization algorithms,'' in \emph{Annual Allerton Conference on
  Communication, Control, and Computing}, 2017, pp. 1206--1212.

\bibitem{hu2017control}
B.~Hu and L.~Lessard, ``Control interpretations for first-order optimization
  methods,'' in \emph{American Control Conference}, 2017, pp. 3114--3119.

\bibitem{hu17b}
B.~Hu, P.~Seiler, and A.~Rantzer, ``A unified analysis of stochastic
  optimization methods using jump system theory and quadratic constraints,'' in
  \emph{Conference on Learning Theory}, vol.~65, 2017, pp. 1157--1189.

\bibitem{hatanaka2018passivity}
T.~Hatanaka, N.~Chopra, T.~Ishizaki, and N.~Li, ``Passivity-based distributed
  optimization with communication delays using {PI} consensus algorithm,''
  \emph{IEEE Transactions on Automatic Control}, vol.~63, no.~12, pp.
  4421--4428, 2018.

\bibitem{hu2018dissipativity}
B.~Hu, S.~Wright, and L.~Lessard, ``Dissipativity theory for accelerating
  stochastic variance reduction: {A} unified analysis of {SVRG} and {K}atyusha
  using semidefinite programs,'' in \emph{International Conference on Machine
  Learning}, 2018, pp. 2043--2052.

\bibitem{han2019systematic}
S.~Han, ``Systematic design of decentralized algorithms for consensus
  optimization,'' \emph{IEEE Control Systems Letters}, vol.~3, no.~4, pp.
  966--971, 2019.

\bibitem{aybat2018robust}
N.~S. Aybat, A.~Fallah, M.~Gurbuzbalaban, and A.~Ozdaglar, ``Robust accelerated
  gradient methods for smooth strongly convex functions,'' \emph{SIAM Journal
  on Optimization}, vol.~30, no.~1, pp. 717--751, 2020.

\bibitem{xiong2020analytical}
H.~Xiong, Y.~Chi, B.~Hu, and W.~Zhang, ``Analytical convergence regions of
  accelerated gradient descent in nonconvex optimization under regularity
  condition,'' \emph{Automatica}, vol. 113, 2020.

\bibitem{mohammadi2020robustness}
H.~Mohammadi, M.~Razaviyayn, and M.~R. Jovanovi{\'c}, ``Robustness of
  accelerated first-order algorithms for strongly convex optimization
  problems,'' \emph{IEEE Transactions on Automatic Control}, vol.~66, no.~6,
  pp. 2480--2495, 2020.

\bibitem{hu2021analysis}
B.~Hu, P.~Seiler, and L.~Lessard, ``Analysis of biased stochastic gradient
  descent using sequential semidefinite programs,'' \emph{Mathematical
  Programming}, vol. 187, no.~1, pp. 383--408, 2021.

\bibitem{gannot2021frequency}
O.~Gannot, ``A frequency-domain analysis of inexact gradient methods,''
  \emph{Mathematical Programming}, pp. 1--42, 2021.

\bibitem{van2017fastest}
B.~Van~Scoy, R.~Freeman, and K.~Lynch, ``The fastest known globally convergent
  first-order method for minimizing strongly convex functions,'' \emph{IEEE
  Control Systems Letters}, vol.~2, no.~1, pp. 49--54, 2017.

\bibitem{cyrus2018robust}
S.~Cyrus, B.~Hu, B.~Van~Scoy, and L.~Lessard, ``A robust accelerated
  optimization algorithm for strongly convex functions,'' in \emph{American
  Control Conference}, 2018, pp. 1376--1381.

\bibitem{fazlyab2018design}
M.~Fazlyab, M.~Morari, and V.~M. Preciado, ``Design of first-order optimization
  algorithms via sum-of-squares programming,'' in \emph{IEEE Conference on
  Decision and Control}, 2018, pp. 4445--4452.

\bibitem{nelson2018integral}
Z.~Nelson and E.~Mallada, ``An integral quadratic constraint framework for
  real-time steady-state optimization of linear time-invariant systems,'' in
  \emph{American Control Conference}, 2018, pp. 597--603.

\bibitem{aybat2019universally}
N.~S. Aybat, A.~Fallah, M.~Gurbuzbalaban, and A.~Ozdaglar, ``A universally
  optimal multistage accelerated stochastic gradient method,'' in
  \emph{Advances in Neural Information Processing Systems}, 2019, pp.
  8525--8536.

\bibitem{michalowsky2020robust}
S.~Michalowsky, C.~Scherer, and C.~Ebenbauer, ``Robust and structure exploiting
  optimisation algorithms: an integral quadratic constraint approach,''
  \emph{International Journal of Control}, pp. 1--24, 2020.

\bibitem{sundararajan2020analysis}
A.~Sundararajan, B.~Van~Scoy, and L.~Lessard, ``Analysis and design of
  first-order distributed optimization algorithms over time-varying graphs,''
  \emph{IEEE Transactions on Control of Network Systems}, vol.~7, no.~4, pp.
  1597--1608, 2020.

\bibitem{hu2019characterizing}
B.~Hu and U.~Syed, ``Characterizing the exact behaviors of temporal difference
  learning algorithms using {M}arkov jump linear system theory,'' in
  \emph{Advances in Neural Information Processing Systems}, 2019, pp.
  8479--8490.

\bibitem{lee2019unified}
D.~Lee and N.~He, ``A unified switching system perspective and convergence
  analysis of {Q}-learning algorithms,'' in \emph{Advances in Neural
  Information Processing Systems}, vol.~33, 2020, pp. 15\,556--15\,567.

\bibitem{borkar2009stochastic}
V.~Borkar, \emph{Stochastic approximation: a dynamical systems
  viewpoint}.\hskip 1em plus 0.5em minus 0.4em\relax Springer, 2009, vol.~48.

\bibitem{borkar2000ode}
V.~Borkar and S.~Meyn, ``The {ODE} method for convergence of stochastic
  approximation and reinforcement learning,'' \emph{SIAM Journal on Control and
  Optimization}, vol.~38, no.~2, pp. 447--469, 2000.

\bibitem{Farahmand2021}
A.~Farahmand and M.~Ghavamzadeh, ``{PID} accelerated value iteration
  algorithm,'' in \emph{International Conference on Machine Learning}, 2021,
  pp. 3143--3153.

\bibitem{Puterman2014}
M.~Puterman, \emph{Markov decision processes: discrete stochastic dynamic
  programming}.\hskip 1em plus 0.5em minus 0.4em\relax John Wiley \& Sons,
  2014.

\bibitem{sutton2018reinforcement}
R.~Sutton and A.~Barto, \emph{Reinforcement learning: An introduction}.\hskip
  1em plus 0.5em minus 0.4em\relax MIT press, 2018.

\bibitem{bertsekas1996neuro}
D.~Bertsekas and J.~Tsitsiklis, \emph{Neuro-dynamic programming}.\hskip 1em
  plus 0.5em minus 0.4em\relax Athena Scientific Belmont, 1996, vol.~5.

\bibitem{bhandari2018finite}
J.~Bhandari, D.~Russo, and R.~Singal, ``A finite time analysis of temporal
  difference learning with linear function approximation,'' in \emph{Conference
  on learning theory}, 2018, pp. 1691--1692.

\bibitem{srikant2019finite}
R.~Srikant and L.~Ying, ``Finite-time error bounds for linear stochastic
  approximation and {TD} learning,'' in \emph{Conference on Learning Theory},
  2019, pp. 2803--2830.

\bibitem{xu2019two}
T.~Xu, S.~Zou, and Y.~Liang, ``Two time-scale off-policy {TD} learning:
  {N}on-asymptotic analysis over {M}arkovian samples,'' in \emph{Advances in
  Neural Information Processing Systems}, 2019.

\bibitem{sun2020finite}
J.~Sun, G.~Wang, G.~B. Giannakis, Q.~Yang, and Z.~Yang, ``Finite-time analysis
  of decentralized temporal-difference learning with linear function
  approximation,'' in \emph{AIstats}, 2020, pp. 4485--4495.

\bibitem{xu2020finite}
P.~Xu and Q.~Gu, ``A finite-time analysis of {Q}-learning with neural network
  function approximation,'' in \emph{International Conference on Machine
  Learning}, 2020, pp. 10\,555--10\,565.

\bibitem{zhang2021finite}
S.~Zhang, Z.~Zhang, and S.~T. Maguluri, ``Finite sample analysis of
  average-reward {TD} learning and {Q}-learning,'' \emph{Advances in Neural
  Information Processing Systems}, vol.~34, 2021.

\bibitem{doan2021finiteb}
T.~T. Doan, ``Finite-time analysis and restarting scheme for linear
  two-time-scale stochastic approximation,'' \emph{SIAM Journal on Control and
  Optimization}, vol.~59, no.~4, pp. 2798--2819, 2021.

\bibitem{Farina2011Pos_sys}
L.~Farina and S.~Rinaldi, \emph{Positive linear systems: theory and
  applications}.\hskip 1em plus 0.5em minus 0.4em\relax John Wiley \& Sons,
  2011.

\bibitem{Rantzer2011}
A.~Rantzer, ``Distributed control of positive systems,'' in \emph{IEEE
  Conference on Decision and Control and European Control Conference}, 2011,
  pp. 6608--6611.

\bibitem{Blanchini2015}
F.~Blanchini, P.~Colaneri, M.~E. Valcher \emph{et~al.}, ``Switched positive
  linear systems,'' \emph{Foundations and Trends{\textregistered} in Systems
  and Control}, vol.~2, no.~2, pp. 101--273, 2015.

\bibitem{Liu2009}
X.~Liu, ``Stability analysis of switched positive systems: a switched linear
  copositive {L}yapunov function method,'' \emph{IEEE Trans. on Circuits and
  Systems II: Express Briefs}, vol.~56, no.~5, pp. 414--418, 2009.

\bibitem{Mason2007}
O.~Mason and R.~Shorten, ``On linear copositive {L}yapunov functions and the
  stability of switched positive linear systems,'' \emph{IEEE Transactions on
  Automatic Control}, vol.~52, no.~7, pp. 1346--1349, 2007.

\bibitem{Xu2018}
Y.~Xu, J.~Dong, R.~Lu, and L.~Xie, ``Stability of continuous-time positive
  switched linear systems: A weak common copositive {L}yapunov functions
  approach,'' \emph{Automatica}, vol.~97, pp. 278--285, 2018.

\bibitem{Fornasini2011}
E.~Fornasini and M.~Valcher, ``Stability and stabilizability criteria for
  discrete-time positive switched systems,'' \emph{IEEE Transactions on
  Automatic control}, vol.~57, no.~5, pp. 1208--1221, 2011.

\bibitem{pastravanu2014max}
O.~C. Pastravanu and M.-H. Matcovschi, ``Max-type copositive {L}yapunov
  functions for switching positive linear systems,'' \emph{Automatica},
  vol.~50, no.~12, pp. 3323--3327, 2014.

\bibitem{costa2006}
O.~Costa, M.~Fragoso, and R.~Marques, \emph{Discrete-time Markov jump linear
  systems}.\hskip 1em plus 0.5em minus 0.4em\relax Springer Science \& Business
  Media, 2006.

\bibitem{costa1993stability}
O.~Costa and M.~Fragoso, ``Stability results for discrete-time linear systems
  with {M}arkovian jumping parameters,'' \emph{Journal of Mathematical Analysis
  and Applications}, vol. 179, no.~1, pp. 154--178, 1993.

\bibitem{el1996robust}
L.~El~Ghaoui and M.~Rami, ``Robust state-feedback stabilization of jump linear
  systems via {LMIs},'' \emph{International Journal of Robust and Nonlinear
  Control}, vol.~6, no. 9-10, pp. 1015--1022, 1996.

\bibitem{Fang2002}
Y.~Fang and K.~Loparo, ``Stochastic stability of jump linear systems,''
  \emph{IEEE Transactions on Automatic Control}, vol.~47, no.~7, pp.
  1204--1208, 2002.

\bibitem{ji1991stability}
Y.~Ji, H.~Chizeck, X.~Feng, and K.~Loparo, ``Stability and control of
  discrete-time jump linear systems,'' \emph{Control-Theory and Advanced
  Technology}, vol.~7, no.~2, pp. 247--270, 1991.

\bibitem{Seiler2003bounded}
P.~Seiler and R.~Sengupta, ``A bounded real lemma for jump systems,''
  \emph{IEEE Transactions on Automatic Control}, vol.~48, no.~9, pp.
  1651--1654, 2003.

\bibitem{DeCarlo2000}
R.~DeCarlo, M.~Branicky, S.~Pettersson, and B.~Lennartson, ``Perspectives and
  results on the stability and stabilizability of hybrid systems,''
  \emph{Proceedings of the IEEE}, vol.~88, no.~7, pp. 1069--1082, 2000.

\end{thebibliography}
\end{document}